\documentclass[letterpaper]{article}

\usepackage{amsmath}
\usepackage{amssymb}
\usepackage{amsthm}
\usepackage{color}

\newcommand{\thalb}{\ensuremath{{\textstyle\frac{1}{2}}}}
\newcommand{\menge}[2]{\big\{{#1}~\big |~{#2}\big\}}

\newcommand{\To}{\ensuremath{\rightrightarrows}}
\newcommand{\scal}[2]{\left\langle{#1},{#2}  \right\rangle}
\newcommand{\RR}{\ensuremath{\mathbb R}}
\newcommand{\rl}{\ensuremath{r_L}}
\newcommand{\all}{\hbox{for all}}
\newcommand{\NN}{\ensuremath{\mathbb N}}
\newcommand{\RX}{\ensuremath{\,\left]-\infty,+\infty\right]}}
\newcommand{\dom}{\ensuremath{\operatorname{dom}}}

\newcommand{\ran}{\ensuremath{\operatorname{ran}}}
\newcommand{\gra}{\ensuremath{\operatorname{gra}}}
\newcommand{\Id}{\ensuremath{\operatorname{Id}}}
\newcommand{\bx}{\ensuremath{\mathbf{x}}}

\newcommand{\rbar}{\,]{-}\infty,\infty]}
\newcommand{\quand}{\quad \hbox{and} \quad}
\newtheorem{theorem}{Theorem}[section]
\newtheorem{lemma}[theorem]{Lemma}
\newtheorem{corollary}[theorem]{Corollary}
\newtheorem{proposition}[theorem]{Proposition}

\theoremstyle{definition}
\newtheorem{definition}[theorem]{Definition}

\newtheorem{example}[theorem]{Example}
\newtheorem{remark}[theorem]{Remark}

\newtheorem{fact}[theorem]{Fact}

\begin{document}

\title{Weak subdifferentials, $r_L$--density and maximal monotonicity}

\author{
Stephen Simons\thanks{
Department of Mathematics, University of California, Santa Barbara, CA\ 93106-3080, U.S.A.
Email: \texttt{simons@math.ucsb.edu}.}
\and
Xianfu Wang\thanks{
Department of Mathematics, University of British Columbia, Kelowna, B.C.\ V1V~1V7, Canada. E-mail: \texttt{shawn.wang@ubc.ca}.}}

\date{}

\maketitle

\begin{abstract}\noindent
In this paper, we first investigate an abstract subdifferential for which  (using Ekeland's variational principle) we can prove an analog of the Br\o ndsted--Rockafellar property.   We introduce the ``$r_L$--density'' of a subset of the product of a Banach space with its dual. A closed $r_L$--dense {\em monotone} set is maximally monotone, but we will also consider the case of {\em nonmonotone} closed $r_L$--dense sets.   As a special case of our results, we can prove Rockafellar's result that the subdifferential of a proper convex lower semicontinuous function is maximally monotone.
\end{abstract}


multifunction, monotonicity, monotone polar, $r_L$--density.


\section{Introduction}

The goal of this paper is to use the function $\rl$ (defined below) to study weak subdifferentials of proper lower semicontinuous functions, approximate Minty type results on multifunctions, and monotone polars of weak subdifferentials.
 
We start the paper by introducing in Definition~\ref{WSUBDIFF} an abstract subdifferential (the {\em weak subdifferential}) that includes that introduced by Thibault and\break Zagrodny in \cite{thibault}, and we establish in Theorem~\ref{ABR} that this weak definition still possesses a  Br\o ndsted--Rockafellar property.

Now let $(E,\|\cdot\|)$ be a real Banach space with topological dual $E^*$. Let $j:=\thalb\|\cdot\|^2$, and  write $J: =\partial j\colon\ E \To E^*$.    The multifunction $J$ is known as the \emph{duality mapping}. 

Define  $r_{L}:E\times E^*\rightarrow\RR$ by
$$(x,x^*)\mapsto \rl(x,x^*):=\thalb\|x\|^2+\thalb\|x^*\|^2+\scal{x}{x^*} = j(x) + j^*(x^*) + \scal{x}{x^*}.$$
It is worth making a few historical comments about the function $r_L$.   It appears explicitly in the ``perfect square criterion for maximality'' in the reflexive case in \cite[Theorem 10.3,\ p.\ 36]{MANDM}.   It also appears explicitly (still in the reflexive case) in Simons--Z\u{a}linescu \cite{simonszal}, with the symbol ``$\Delta$'', where it was used to study Fitzpatrick functions and the maximality of a sum of monotone operators.   It was used in the nonreflexive case by Zagrodny in \cite{ZAGRODNY}.    

Obviously,
\begin{align*}
\all\ (x,x^*) \in E \times E^*,\quad \rl(x,x^*) & \geq \thalb\|x\|^2+\thalb\|x^*\|^2-\|x\|\|x^*\|\geq 0
\end{align*}
and
\begin{equation}\label{e:duality}
\rl(x,x^*)=0 \quad \Leftrightarrow \quad x^*\in -J(x).
\end{equation}

\begin{definition}\label{RLDENSE}
Let $A\subseteq E\times E^*$. We say that $A$ is {\em$\rl$--dense in $E\times E^*$} if, for all $(y,y^*) \in E \times E^*$, 
$$\inf_{(s,s^*)\in A}\rl(s - y,s^* - y^*)=0.$$
We say that $A$ is {\rm stably $\rl$--dense in $E\times E^*$} if, for all $(y,y^*) \in E \times E^*$, there exists $M \ge 0$ such that
$$\inf_{(s,s^*)\in A,\ \|s - y\| \le M,\ \|s^* - y^*\| \le M}\rl(s - y,s^* - y^*)=0.$$
\end{definition}

The concept of $r_L$--density was studied in the context of monotone operators (and even the more general situation of ``$L$--positive sets'') in \cite{simonsprep}.   That paper also contains the motivation for the notation ``$r_L$''.

In Theorem~\ref{t:main}, we prove that the graph of the weak subdifferential of a proper lower semicontinuous function is stably $r_L$--dense provided that the function is not too wildly negative.   If $f$ is convex, this result generalizes\break Rockafellar's theorem on the maximal monotonicity of subdifferentials.

In Sections \ref{MINTYsec}--\ref{POLsec}, we give some approximate Minty type results (some of them for Hilbert spaces) and some results on {\em monotone polarity}, and in Section~\ref{RLDsec}, we study sufficient conditions for a multifunction to be stably $r_L$--dense.

We finish this introduction with some notation.   If $S:E\To E^*$, we write $\gra S$, $\dom S$ and $\ran S$ for the {\em graph}, {\em domain} and {\em range} of $S$, which are defined by $\gra S := \menge{(x,x^*)\in E\times E^*}{x^*\in Sx}$, $\dom S:=\menge{x\in E}{S(x)\neq\varnothing}$ and $\ran S:=\bigcup\menge{S(x)}{x\in\dom S}$.
 
\begin{lemma}\label{EXACTMIN}
Let $S:\ E\To E^*$ be a multifunction.
Then
$$\gra S - \gra (-J) = E\times E^*$$
if and only if, for all $(y,y^*)\in E\times E^*$,
$$\min_{(s,s^*)\in \gra S}\rl(s - y,s^* - y^*)=0.$$
\end{lemma}
\begin{proof}
This is immediate from \eqref{e:duality} and the definition of $\gra S$.
\end{proof}
\section{Weak subdifferentials and a Br\o ndsted--\break Rockafellar property}
If $f:E\rightarrow\RX$ is proper, convex and lower semicontinuous $\partial f$ is the classical subdifferential of convex analysis, defined by
$$\partial f(x)=\menge{x^*\in E^*}{\all\ y \in E,\ f(y) \ge f(x) + \scal{y-x}{x^*}}.$$
\begin{definition}\label{WSUBDIFF}
A {\em weak subdifferential}, $\partial_w$, is a rule that associates with each proper lower semicontinuous function $f:E\rightarrow\RX$ a multifunction\break $\partial_w f:E\To E^*$ such that
\begin{enumerate}
\item $0\in\partial_w f(x)$ if $f$ attains a strict global minimum at $x$.
\item $\partial_w (f+h)(x) \subseteq \partial_w f(x) + \partial h(x)$ whenever $h$ is a continuous convex real function on $E$.
\end{enumerate}
\end{definition}
The abstract subdifferential introduced by Thibault and Zagrodny in \cite{thibault} gives a weak subdifferential.   This implies that a number of other subdifferentials that have been introduced over the years also give weak subdifferentials.   See the list on \cite[p. 35]{thibault}.   In particular, the Clarke-Rockafellar subdifferential is a weak subdifferential \big(see \cite[Section 2.9]{clarke}, or a combination of \cite[Theorem 5]{Rock80} and \cite[Theorem 2]{Rock79}\big).   Also, \cite[Corollary 4.3]{shao} or \cite{Boris1} show that Mordukhovich's limiting subdifferential is a weak subdifferential if we confine our attention to Asplund spaces.  
\medbreak
We will use the following well known result from variational analysis.   See, for instance, \cite[Lemma 3.13, p. 45]{phelpsbook} or \cite[Theorem 1.45, p. 19]{BC2011}. 
\begin{fact}[Ekeland]\label{f:variational}
Assume that $g:E\rightarrow\RX$ is a proper lower semicontinuous function and bounded below. Suppose that $\alpha,\beta>0$ and $g(u)\leq \inf_{E}g + \alpha\beta$. Then there exists
$s\in \dom g$ such that
\begin{enumerate}
\item $g(s) + \beta\|s- u\|\leq g(u)$;
\item $g(x)+\beta \|x-s\|> g(s)$ whenever $x\neq s$.
\end{enumerate}
We note from (i) above that $g(u) - \alpha\beta + \beta\|s - u\|\leq g(u)$, consequently $\|s-u\|\leq \alpha$.
\end{fact}
We will need the following simple result:
\begin{lemma}\label{SIMPLE}
Let $E$ be a Banach space, $s \in E$, $\beta > 0$ and $h\colon\ E \to \RX$ be defined by $h := \beta\|\cdot - s\|$.   Then $\partial h(s) = \{z^* \in E^*\colon\ \|z^*\| \le \beta\}$.
\end{lemma}
\begin{proof}
By definition, $z^* \in \partial h(s)$ exactly when, for all $x \in E$, $h(s) + \scal{x}{z^*} \le h(x + s)$. But, from the definition of $h$, this is equivalent to saying that, for all $x \in E$, $\scal{x}{z^*} \le \beta\|x\|$, that is to say $\|z^*\| \le \beta$.
\end{proof}

\begin{lemma}\label{BRLEM}
Let $E$ be a Banach space, $g\colon\ E \to \rbar$ be lower semicontinuous, $\alpha,\beta > 0$, $u \in \dom g$ and $g(u) \le \inf_Eg + \alpha\beta$.   Then there exists $(s,x^*) \in \gra \partial_w g$ such that $\|s - u\| \le \alpha$, $g(s) \le g(u)$ and $\|x^*\| \le \beta$.
\end{lemma}
\begin{proof}
Define $h$ as in Lemma \ref{SIMPLE}, and let $s$ be as in Fact \ref{f:variational}.   Then Fact \ref{f:variational}(ii) implies that $g + h$ attains a strict global minimum at $s$, and Definition \ref{WSUBDIFF}(i) implies that $0 \in \partial_w(g + h)(s)$.   Since $h$ is convex and continuous, Definition \ref{WSUBDIFF}(ii) implies that $0 \in \partial_w g(s) + \partial h(s)$.   Thus, from Lemma \ref{SIMPLE}, there exists $z^* \in E^*$ such that $\|z^*\| \le \beta$ and $0 \in \partial_w g(s) + z^*$.   The result follows by taking $x^* := -z^*$.
\end{proof}
\smallbreak
We now prove a Br\o ndsted--Rockafellar property for weak subdifferentials.   Let $f\colon\ E \to \rbar$ be proper.  Following, e.g., \cite[p. 75]{zalinescu},  we define the \emph{Fenchel conjugate} of $f$, $f^*\colon\ E^* \to \rbar$ by $f^*(x^*) := \sup\big[x^* - f\big]$.

\begin{theorem}\label{ABR}
Let $E$ be a Banach space, $\alpha,\beta > 0$, $f:E \to \RX$ be proper and lower semicontinuous, $(u,u^*) \in E \times E^*$ and
\begin{equation}\label{ALPHAHBETA}
f(u) + f^*(u^*) \le \scal{u}{u^*} + \alpha\beta.
\end{equation}
Then there exists $(s,s^*) \in \gra \partial_w f$ such that $\|s - u\| \le \alpha$, $\|s^* - u^*\| \le \beta$ and $f(s) - \scal{s}{u^*} \le f(u) - \scal{u}{u^*}$.
\end{theorem}

\begin{proof}From \eqref{ALPHAHBETA}, $(f - u^*)(u) \le \inf_E\big[f - u^*\big] + \alpha\beta$ and so Lemma \ref{BRLEM} gives $(s,x^*) \in \gra \partial_w(f - u^*)$ such that $\|s - u\| \le \alpha$ and $\|x^*\| \le \beta$.   From Definition \ref{WSUBDIFF}(ii), $\partial_w(f - u^*) \subseteq \partial_wf - u^*$, and so there exists $s^* \in \partial_wf(s)$ such that $x^* = s^* - u^*$.   This gives the desired result.
\end{proof}

\section{The $r_L$--density of the weak subdifferentials of certain functions}\label{RLDSUB}

We suppose that $\partial_w$ is a weak subdifferential as defined in Definition \ref{WSUBDIFF}.

\begin{definition}\label{NODOWN}
Let $f:E\rightarrow\RX$.   We say that $f$ has {\em insignificant downside} if there exist $a_0,b_0,c_0 \in \RR$ with $a_0 < \thalb$ and, for all $x \in E$, $f(x) \ge -a_0\|x\|^2 - b_0\|x\| - c_0$.
\end{definition}
\begin{theorem}\label{t:main}
Assume that $f:E\rightarrow\RX$ is proper, lower semicontinuous and has insignificant downside.  Then $\gra \partial_w f$ is stably $\rl$--dense in $E\times E^*$.
\end{theorem}

\begin{proof} Let $(y,y^*) \in E \times E^*$ and $k := j(\cdot - y) - y^*$.   We note for future reference that, for all $s \in E$,
\begin{equation}\label{PARTIALK}
\partial k(s) = \partial j(s - y) - y^* = J(s - y) - y^*.
\end{equation}
Let $a_0,b_0,c_0$ be as in Definition \ref{NODOWN}, and write $a := \thalb - a_0 > 0$, $b := \|y\| + \|y^*\| + b_0$ and $c := c_0 - j(y)$.   Then, for all $x\in E$,
\begin{align*}
(f + k)(x)&\geq \thalb\|x - y\|^2 - \|y^*\|\|x\| - a_0\|x\|^2 - b_0\|x\| - c_0\\
&\geq \thalb\big(\|x\| - \|y\|\big)^2 - \|y^*\|\|x\| - a_0\|x\|^2 - b_0\|x\| - c_0
\end{align*}
thus
\begin{equation}\label{e:biggerx}
(f + k)(x)\geq a\|x\|^2 - b\|x\| - c \ge \min_{\lambda\in\RR}\big[a\lambda^2 - b\lambda - c\big] \in \RR.
\end{equation}
Let $m:=\inf_{E}(f + k) \in \RR$ and $M:=b/2a +\sqrt{b^2 + 4a(c + m + 1)}/2a + \|y\| + 2$. Now let $\varepsilon > 0$.  Choose $\beta\in ]0,1]$ such that $2M\beta<\varepsilon$. Then there exists $u \in E$ such that
\begin{equation}\label{e:almostmin}
(f + k)(u) \leq \inf_E(f + k) + \beta \leq m +1.
\end{equation}
\eqref{e:biggerx} and \eqref{e:almostmin} imply that $a\|u\|^2 - b\|u\| \le c + m + 1 $.   Thus, completing the square,
\begin{equation}\label{NORMU}
\|u\|\leq b/2a +\sqrt{b^2 + 4a(c + m + 1)}/2a = M - \|y\| - 2.
\end{equation}
Applying Lemma \ref{BRLEM} to \eqref{e:almostmin}, with $g := f + k$ and $\alpha = 1$,  gives $s\in E$ and $x^*\in\partial_w (f+k)(s)$ such that
\begin{equation}\label{e:notfar}
\|s-u\|\leq 1 \quand \|x^*\|\leq \beta.
\end{equation}
Combining this with \eqref{NORMU},
\begin{equation}\label{e:notfar1}
\|s - y\| \leq \|s-u\| + \|u\| + \|y\| \leq 1 + (M - 2) = M - 1.
\end{equation}
From Definition \ref{WSUBDIFF}(ii) and \eqref{PARTIALK},
$$\partial_w(f+ k)(s)\subseteq \partial_w f(s) + \partial k(s) = \partial_w f(s) + J(s - y) - y^*,$$
so there exists $s^*\in\partial_w f(s)$ such that $x^* - (s^* - y^*) \in J(s - y)$. Consequently, from the well known properties of $J$,
\begin{equation}\label{e:notfar2}
\scal{s - y}{x^* - (s^* - y^*)}=\|s - y\|^2 \quand \|(s^* - y^*) - x^*\|=\|s - y\|.
\end{equation}
It follows that $\scal{s - y}{s^* - y^*}=\scal{s - y}{x^*} - \|s - y\|^2$, and so 
\begin{equation}\label{SSTAR}
\thalb\|s - y\|^2 + \scal{s - y}{s^* - y^*}=\scal{s - y}{x^*} - \thalb\|s - y\|^2
\end{equation}
thus, from \eqref{e:notfar}, \eqref{e:notfar1} and \eqref{e:notfar2},
\begin{equation}\label{NORMSSTAR}
\|s^* - y^*\|\leq \|(s^* - y^*) - x^*\|+\|x^*\|\leq \|s - y\|+\beta \leq (M - 1) + 1 = M.
\end{equation}
It is clear from \eqref{e:notfar}, \eqref{e:notfar1}, \eqref{SSTAR} and \eqref{NORMSSTAR} that
\begin{align*}
r_L(s - y,s^* - y^*) &= \thalb \|s - y\|^2+\scal{s - y}{s^* - y^*} + \thalb\|s^* - y^*\|^2\\
&\leq \scal{s - y}{x^*}-\thalb \|s - y\|^2+\thalb (\|s - y\|+\beta)^2\\
&\leq 2\|s - y\|\beta +\thalb \beta^2 \leq 2\beta (M - 1) + \beta \leq 2M\beta<\varepsilon.
\end{align*}
The result now follows since, from \eqref{e:notfar1} and \eqref{NORMSSTAR}, $\|s - y\| \le M$ and $\|s^* - y^*\| \le M$.
\end{proof}

\begin{remark}
If $f$ is bounded below by a continuous affine functional then obviously $f$ has insignificant downside.   The proof of Theorem~\ref{t:main} is patterned after the proof in \cite[Theorem 9.3]{simonsprep} that the graph of the subdifferential of a proper, convex lower semicontinuous function is stably $r_L$--dense. 
\end{remark}

\begin{corollary}\label{FINDIM}
Assume that $E$ is finite dimensional, $\gra \partial_wf$ is closed and\break $f:E\rightarrow\RX$ is proper, lower semicontinuous and has insignificant downside.  Then $\gra \partial_w f - \gra (-J) = E\times E^*$ and $\ran(\partial_w f + J) = E^*$.
\end{corollary}
\begin{proof}
Let $(y,y^*) \in E \times E^*$.   Theorem \ref{t:main} provides a bounded sequence\break $(s_n,s_n^*)_{n\in\NN}$ of elements of $\gra \partial_w f$ such that, for all $n \ge 1$, $\rl(s_n - y,s^*_n - y^*) < 1/n$, and the Bolzano-Weierstrass Theorem gives a  subsequence $(s_{n_k},s_{n_k}^*)_{k\in\NN}$ of $(s_n,s_n^*)_{n\in\NN}$ and $(s,s^*) \in E \times E^*$ such that $(s_{n_{k}},s_{n_{k}}^*) \to (s,s^*)$.   Since $\gra\partial_w f$ is closed, $(s,s^*) \in \gra\partial_w f$, and obviously $\rl(s - y,s^* - y^*)=0$.   It now follows from Lemma \ref{EXACTMIN} that $\gra \partial_w f - \gra (-J) = E\times E^*$.   Now let $z^*$ be an arbitrary element of $E^*$.   From what we have just proved, there exists $(s,s^*) \in \gra\partial_w f$ and $(x,x^*) \in \gra(-J)$ such that $(s - x,s^* - x^*) = (0,z^*)$.   This implies that $x = s$, and so $x^* \in -J(x) = -J(s)$.  Thus $z^* = s^* - x^* \in \partial_w f(s) + J(s) = (\partial_wf+ J)(s)$.   This completes the proof that $\ran(\partial_w f + J) = E^*$.
\end{proof}

\begin{example}\label{RCASE}
In this example, we suppose that $E = \RR$ and that $\partial_w$ has the special property that, whenever $f$ is a polynomial, $\partial_wf(x) = \{f^\prime(x)\}$.   For instance, $\partial_w$ could be the Clarke--Rockafellar subdifferential.  We note that, for all $x \in \RR$, $J(x) = \{x\}$ and, for all $(x,x^*) \in \RR \times \RR$, $r_L(x,x^*) = \thalb(x + x^*)^2$.
\begin{enumerate}
\item If $\lambda < \thalb$ and $f(x) := -\lambda x^2$ then $f$ has insignificant downside, and so Corollary \ref{FINDIM} can be applied.
\item If $f(x) := -\thalb x^2$ then, for all $x \in \RR$, $(\partial_w f + J)(x) = \{-x + x\} = \{0\}$.   Thus $\ran(\partial_w f + J) \ne \RR$.    Thus the second conclusion of Corollary~\ref{FINDIM} fails and, working backwards, the conclusion of  Theorem~\ref{t:main} also fails.
\par
\item If $\lambda > \thalb$ and $f(x) := -\lambda x^2$ then $f$ does not have significant downside.   Nevertheless, given $(y,y^*) \in \RR \times \RR$, let $s = (y + y^*)/(1 - 2\lambda)$.   Then $(s,-2\lambda s) \in \gra \partial_w f$ and $r_L(s - y,-2\lambda s - y^*) = \thalb(s - y -2\lambda s - y^*)^2 = 0$, so $\gra \partial_w f$ is stably $r_L$--dense.
\par
\item Let $n$ be an odd integer, $n \ge 3$ and $f(x) := x^n$.   Obviously $f$ does not have insignificant downside.   Note that $(\partial_w f + J)(x) = \{nx^{n - 1} + x\}$.   Since $n - 1$ is an even integer and $n - 1 \ge 2$,     $\ran(\partial_w f + J) \ne \RR$.   Thus the second conclusion of Corollary~\ref{FINDIM} fails and, working backwards, the conclusion of  Theorem~\ref{t:main} also fails.
\end{enumerate}
\end{example}
\section{Approximate Minty type results}\label{MINTYsec}
We start this section by recalling the classical Br\o ndsted--Rockafellar theorem, which can actually be deduced from Theorem~\ref{ABR}.
\begin{fact}[See {\cite[pp. 608--609]{BRON}}]\label{BRfact}
Let $E$ be a Banach space, $\alpha,\beta > 0$, $f:E \to \RX$ be proper, convex and lower semicontinuous, $(u,u^*) \in E \times E^*$ and
$$f(u) + f^*(u^*) \le \scal{u}{u^*} + \alpha\beta.$$
Then there exists $(t,t^*) \in \gra \partial f$ such that $\|t - u\| \le \alpha$ and $\|t^* - u^*\| \le \beta$.
\end{fact}
\begin{theorem}\label{APPMINthm}
Let $E$ be a Banach space and $A$ be an $r_L$--dense subset of $E \times E^*$. Then $A - \gra (-J)$ is dense in $E \times E^*$.
\end{theorem}
\begin{proof}
Let $(y,y^*) \in E \times E^*$ and $\varepsilon > 0$.   By hypothesis, there exists $(s,s^*)\in A$ such that $\rl(s - y,s^* - y^*) \le \varepsilon$, which can be rewritten:
$$j(s - y) + j^*(y^*- s^*) \le \scal{s - y}{y^*- s^*} + \varepsilon.$$
From Fact~\ref{BRfact} with $f := j$, $u := s - y$ and $u^* := y^* - s^*$, there exists $(t,t^*) \in \gra J$ such that
$$\|(s - t) - y\| = \|t - (s - y)\| \le \sqrt{\varepsilon}$$
and
$$\|(s^* + t^*) - y^*\| = \|t^* - (y^* - s^*)\| \le \sqrt{\varepsilon}.$$
Thus
$$\|(s - t,s^* + t^*) - (y,y^*)\|^2 = \|(s - t) - y\|^2 + \|(s^* + t^*) - y^*\|^2 \le 2\varepsilon.$$
Now $(t,-t^*) \in G(-J)$, and so $(s - t,s^* + t^*) \in A - \gra (-J)$.   Since $\varepsilon$ can be made arbitrarily small, it follows that $A-\gra (-J)$ is dense in $E \times E^*$.
\end{proof}
\begin{remark}\label{APPMINrem}
We do not know if the converse of Theorem~\ref{APPMINthm} is true.   In other words, if $A - \gra (-J)$ is dense in $E \times E^*$ then does it follow that $A$ is an $r_L$--dense subset of $E \times E^*$?   It is worth pointing out that if $A - \gra (-J) = E \times E^*$ then, for all $(y,y^*) \in E \times E^*$, there exists $(s,s^*) \in A$ and $(t,t^*) \in \gra(J)$ such that $(s - t,s^* + t^*) = (y,y^*)$.   But then $(s - y,y^* - s^*) = (t,t^*) \in \gra(J)$, and so $j(s - y) + j^*(y^* - s^*) = \scal{s - y}{y^* - s^*}$, that is to say $r_L(s - y,s^* - y^*) = 0$.   There are certain technical problems proving the ``approximate'' version of this.
\end{remark}

For the rest of this section, we examine the special results that are true in the Hilbert space case.

\begin{proposition}\label{CLOSEDprop}
Let $H$ be a Hilbert space, and $S:H \To H$ be a set-valued mapping.  Then $\ran(S + \Id)=H$ if, and only if,
\begin{equation}\label{MINCON}
\all\ (y,y^*)\in H\times H,\ \min_{(s,s^*)\in \gra(S)}
\rl(s - y,s^* - y^*)=0.
\end{equation}
\end{proposition}
\begin{proof}
``$\Rightarrow$'' Assume that $\ran (S + \Id) = H$. Let $(y,y^*)\in H\times H$.   Since $y + y^* \in H$, there exists $(s,s^*) \in \gra(S)$ such that $s^* + s =y + y^*$, so $s-y = -(s^*-y^*)$. We have
\begin{align*}
\rl(s-y,s^*-y^*)
&= \thalb\|s-y\|^2+\thalb\|s^*-y^*\|^2+\scal{s-y}{s^*-y^*}\\
&= \|s-y\|^2-\|s-y\|^2=0.
\end{align*}

``$\Leftarrow$" Assume that \eqref{MINCON} is satisfied and $y^* \in H$. Then, from \eqref{MINCON} with $y = 0$, there exists $(s,s^*) \in \gra(S)$ such that $\rl(s,s^* - y^*)=0$, that is,
$$0=\thalb\|s\|^2+\thalb\|s^* - y^*\|^2 + \scal{s}{s^* - y^*}=
\thalb\|s + s^* - y^*\|^2.$$
Thus, $y^* = s^* + s \in \ran(S + \Id)$. Since $y^* \in H$ was arbitrary,
we conclude that $\ran(S + \Id) = H$.
\end{proof}

\begin{proposition}\label{p:almostclose}
Let $H$ be a Hilbert space, and $S:H \To H$ be a set-valued mapping.   Then $\ran(S + \Id)$ is dense in $H$ if, and only if, 
\begin{equation}\label{INFCON}
\all\ (y,y^*)\in H\times H,\  \inf_{(s,s^*)\in \gra(S)}
\rl(s - y,s^* - y^*)=0.
\end{equation}
\end{proposition}
\begin{proof}
``$\Rightarrow$" Assume that $\ran(S + \Id)$ is dense in $H$, $(y,y^*) \in H\times H$ and $\varepsilon>0$. Since $y + y^*\in H$, we see that there exists $(s,s^*) \in \gra(S)$ such that $\|s^* + s - (y + y^*)\|<\sqrt{2\varepsilon}$.
We have
\begin{align*}
\rl(s-y,s^*-y^*) &=\thalb\|s-y\|^2+\thalb\|s^*-y^*\|^2+\scal{s-y}{s^*-y^*}\\
& =
\thalb\|s-y+s^*-y^*\|^2=\thalb\|s^* + s - (y + y^*)\|^2<\varepsilon.
\end{align*}
Since $\varepsilon>0$ was arbitrary,
we have $\inf_{(y,y^*)\in \gra(S)}
\rl(s-y,s^*-y^*)=0.$

``$\Leftarrow$" Assume that \eqref{INFCON} is satisfied, $y^* \in H$ and $\varepsilon>0$. Then, from \eqref{INFCON} with $y = 0$, there exists $(s,s^*) \in \gra(S)$ such that $\rl(s,s^* - y^*)<\varepsilon$, that is,
\begin{align*}
\varepsilon & >\thalb\|s\|^2+\thalb\|s^*-y^*\|^2+\scal{s}{s^*-y^*} = \thalb\|s + s^*-y^*\|^2.
\end{align*}
Since $s + s^* = s^* + s \in \ran(A+\Id)$, $y^* \in H$ and $\varepsilon>0$ was arbitrary, we conclude that $\ran(A+\Id)$ is dense in $H$.
\end{proof}

\begin{remark}
We give a simple example to illustrate the difference between the conditions of Propositions~\ref{CLOSEDprop} and \ref{p:almostclose}.   If $\bx = (x_n)_{n\in \NN} \in \ell^2$, define\break $S\bx := -\bx + (x_n/n)_{n\in \NN} \in \ell^2$.   Then $(S + \Id)\bx = (x_{n}/n)_{n\in \NN}$.   Obviously, $(1/n)_{n\in\NN}\in\ell^2\setminus \ran(S+\Id)$ but, by considering finitely nonzero sequences, we can see that $\ran(S + \Id)$ is dense in $\ell^2$. 
\end{remark}

\section{Monotone polars}\label{POLsec}
Let $A \subseteq E \times E^*$ and $(x,x^*) \in E \times E^*$.   Following Phelps, \cite{PRAGUE}, we say that $(x,x^*)$ is {\em monotonically related to} $A$ when, for all $(s,s^*) \in A$, $\scal{s - x}{s^* - x^*}\geq 0$.   The {\em monotone polar of} $A$ is the set of all elements of $E \times E^*$ that are monotonically related to $A$.   This set has been introduced by various authors over the years.   It appears in \cite[p.\ 191]{RANGE} in a more abstract version;  it appears in Mart\'\i nez-Legaz--Svaiter, \cite[p.\ 32]{MLS} under the notation $A^\mu$; it also appears in \cite[p.\ 1737]{lassonde} under the notation $A^0$.  We will use the notation $A^\mu$ since the other notation has so many different meanings.   Thus we have  
$$A^\mu := \menge{(x,x^*)\in E\times E^*}{\all\ (s,s^*)\in A,\ \scal{s - x}{s^* - x^*} \ge 0}.$$
Then, of course, $A$ is \emph{monotone} exactly when $A \subseteq A^\mu$, and $A$ is \emph{maximally monotone} exactly when $A = A^\mu$.

\begin{lemma}\label{BARlem}
Let $A$ be an $r_L$--dense subset of $E \times E^*$.   Then, writing $\overline{A}$ for the norm--closure of $A$, $A^\mu \subseteq \overline{A}$.
\end{lemma}

\begin{proof}
Let $(x,x^*) \in A^\mu$ and $\varepsilon > 0$.   Since $A$ is $\rl$-dense, there exists $(s,s^*)\in A$ such that
$$\thalb\|s - x\|^2+\thalb\|s^* - x^*\|^2 + \scal{s - x}{s^* - x^*} = r_L\big((s,s^*) - (x,x^*)\big) < \varepsilon.$$
Now $(x,x^*) \in A^\mu$ and $(s,s^*)\in A$, and so $\scal{s - x}{s^* - x^*} \ge 0$.   Thus
$$\thalb\|(s,s^*) - (x,x^*)\|^2 = \thalb\|s - x\|^2 + \thalb\|s^* - x^*\|^2 < \varepsilon.$$
Since $\varepsilon>0$ was arbitrary, $(x,x^*) \in \overline{A}$.
\end{proof}

\begin{theorem}\label{CLOSEDRLMAXthm}
Let $S\colon\ E \To E^*$ and $\gra S$ be a closed, $\rl$-dense {\em monotone} subset of $E\times E^*$.   Then $S$ is maximally monotone.
\end{theorem}
\begin{proof}
From Lemma \ref{BARlem}, $(\gra S)^\mu \subseteq \overline{\gra S}$.  Consequently, since $\gra S$ is closed, $(\gra S)^\mu \subseteq \gra S$.  The opposite inclusion follows from the monotonicity of $S$, and so $(\gra S)^\mu = \gra S$.   Consequently, $\gra S$ is a maximally monotone subset of $E \times E^*$, from which $S$ is maximally monotone.
\end{proof}

\begin{remark}
If $S\colon\ E \To E^*$ is maximally monotone and either $\ran S = E^*$ or $E$ is reflexive then $\gra S$ is $r_L$--dense in $E \times E^*$.   See \cite[Theorems~6.5(b) and 6.6(b)]{simonsprep}.   In the latter case, in fact $S$ is maximally monotone if and only if for every $(x,x^*)\in E\times E^*$, $\min_{(s,s^*)\in \gra S}\rl(s - x,s^* - x^*)=0$.   See \cite[Theorem~10.3, p.\ 36]{MANDM}.   If $S\colon\ E \To E^*$ is maximally monotone and $\dom S = E$ then it does {\em not} follow that $\gra S$ is $r_L$--dense in $E \times E^*$, even if $S$ is single--valued and linear:   Define $S\colon\ \ell^1 \mapsto \ell^\infty = E^*$ by $(Sx)_n = \sum_{k \ge n} x_k$ ($S$ is the ``tail'' operator).  \big(See \cite[Example~11.5, pp.\ 283--284]{BR}.\big)   
\end{remark}

\begin{theorem}[On the monotone polar of subdifferentials]\label{t:monoclosure}
Let $\partial_w$ be a weak subdifferential as defined in Definition \ref{WSUBDIFF}, and $f:E\rightarrow\RX$ be proper, lower semicontinuous and have insignificant downside.  Then
$$\big(\gra\partial_w f\big)^\mu \subseteq \overline{\gra \partial_w f}.$$
\end{theorem}

\begin{proof}
This is immediate from Lemma~\ref{BARlem} and Theorem~\ref{t:main}.
\end{proof}
\medbreak
In what follows, we write $\partial_{CR}$ for the the Clarke-Rockafellar subdifferential.

\begin{corollary}\label{LLcor}
Let $f:E\rightarrow\RX$ be a locally Lipschitz function with insignificant downside.   Then
$$\big(\gra\partial_{CR}f\big)^\mu\subseteq \gra\partial_{CR}f.$$
\end{corollary}

\begin{proof}
When $f$ is locally Lipschitz, it follows from, e.g., \cite[Proposition 2.1.5(b)]{clarke} or \cite[Theorem 5.2.7]{bzhu} that $\gra\partial_{CR}f$ is closed, and we obtain the result from Theorem~\ref{t:monoclosure}.
\end{proof}

\medbreak
The following example shows that, in the situation of Corollary~\ref{LLcor},  $\gra\big(\partial_{CR}f\big)^\mu$ might be empty.

\begin{example}
Let $f(x)=\sin x$ on $\RR$. Then, for all $x \in \RR$, $\partial_{CR}f(x)=\{\cos x\}$.   If $(x,x^*)\in \big(\gra \partial_{CR}f\big)^\mu$ then, by Theorem~\ref{t:monoclosure}, $(x,x^*)\in \gra \partial_{CR}f$, and so $x^*=\cos x$ and
$$\all\ y \in \RR,\ \scal{y-x}{\cos y-\cos x}\geq 0.$$
Thus, whenever $y\geq x$, $\cos y \geq \cos x$, from which $\cos x=-1$;
also, whenever $y\leq x$, $\cos y \leq \cos x$, from which $\cos y\leq -1$ whenever $y\leq x$. Since this is impossible,  $(\gra \partial_{CR}f)^\mu =\varnothing$.
\end{example}

\medbreak
One immediate consequence of Theorem~\ref{t:monoclosure} is the following celebrated result due to Rockafellar \cite{Rock70}.
\begin{theorem}
Let $f:E\rightarrow\RX$ be a proper lower semicontinuous convex function. Then $\partial f\colon E\To E^*$ is maximally monotone.
\end{theorem}
\begin{proof} It is well known that $\gra \partial f$ is closed in $E\times E^*$ and (from a separation theorem in $E\times\RR$) that there exists $u^*\in E^*$ and $\gamma\in\RR$ such that $f\geq u^*+\gamma$ on $E$, consequently $f$ has insignificant downside.   We now apply Theorem~\ref{t:monoclosure} with $\partial_w = \partial_{CR}$ and (noting from \cite[Theorem 5]{Rock80}  that $\partial_{CR}f = \partial f$)  derive that
$$\big(\gra\partial f\big)^\mu = \big(\gra\partial_{CR}f\big)^\mu \subseteq \overline{\gra \partial_{CR}f} = \overline{\gra \partial f} = \gra \partial f.$$
The result now follows since $\partial f$ is monotone, and so $\gra \partial f \subseteq \big(\gra\partial f\big)^\mu $. 
\end{proof}
\section{$\rl$--density of multifunctions}\label{RLDsec}

\begin{definition} Let $E$ be a Banach space and $T:E\rightrightarrows E^*$.   We say that $T$ has {\em hyperdense range} if, for every $y^*\in E^*$ there exists a sequence $\big\{(t_n,t_n^*)\big\}_{n \ge 1}$ of elements of $\gra T$
such that
$$\sup_{n \ge 1}\|t_n\| < \infty\quad\hbox{and}\quad \quad \lim_{n\rightarrow\infty}\|t_n^* - y^*\|=0.$$
If $T$ is surjective then obviously $T$ has hyperdense range, and in this case\break Theorem~\ref{t:rloperator} can be rewritten in the simpler form: {\em if $\gra S + \gra (-J) = E \times E^*$ then $S$ is stably $\rl$--dense in $E\times E^*$.}   In the case when $S$ is monotone and $E$ is reflexive, we obtain a generalization of  \cite[Theorem~10.3($\Longleftarrow$), p.\ 36]{MANDM}.
\end{definition}
\begin{theorem}\label{t:rloperator}
Let $S: E \rightrightarrows E^*$.   Assume that for every $y\in E$, the mapping $S + J(\cdot-y)$ has hyperdense range.   Then $S$ is stably $\rl$--dense in $E\times E^*$.
\end{theorem}
\begin{proof} Let $(y,y^*)\in E\times E^*$ and $0 < \varepsilon < 1$.   By hypothesis,  there exist a sequence $\big\{(s_n,t_n^*)\big\}_{n \ge 1}$ of elements of $E \times E^*$ and $M > 0$ such that, for all $n \ge 1$,
\begin{equation}\label{e:bound1}
t_n^*\in S(s_n) + J(s_n - y),\quad \|s_n\|\leq M,\quand \lim_{n \to \infty}\|t_n^*-y^*\| = 0.
\end{equation}
Choose
$$0<\beta<\frac{\varepsilon}{2(M + \|y\|)+1}<1.$$
Then, from \eqref{e:bound1}, there exists $(s,t^*) \in E \times E^*$ such that
\begin{equation}\label{e:bound2}
t^*\in S(s) + J(s - y),\quad \|s\|\leq M,\quand \|t^*-y^*\| < \beta.
\end{equation}
Thus there exists $(s,s^*) \in \gra S$ such that $t^* - s^* \in J(s - y)$.   By the properties of duality mappings,
\begin{equation}\label{e:propJ}
\scal{s - y}{t^* - s^*}=\|s - y\|^2\quand \|t^* - s^*\|=\|s - y\|.
\end{equation}
It follows that
\begin{align*}
\scal{s - y}{s^* - y^*}
&= \scal{s - y}{t^* - y^*} - \scal{s - y}{t^* - s^*}\\
&= \scal{s - y}{t^* - y^*} - \|s - y\|^2,
\end{align*}
which gives
\begin{equation}\label{e:innerdiff}
 \thalb\|s - y\|^2 + \scal{s - y}{s^* - y^*}
= \scal{s - y}{t^* - y^*} - \thalb\|s - y\|^2.
\end{equation}
Now, from \eqref{e:bound2},
\begin{equation}\label{e:preimage}
\|s - y\|\leq \|s\|+\|y\|\leq M + \|y\|,
\end{equation}
and, combining this with \eqref{e:bound2} and \eqref{e:propJ},
\begin{equation}\label{e:image}
\|s^* - y^*\| \leq \|t^* - s^*\|+\|t^* - y^*\|
\leq \|s - y\|+\beta  \leq M+\|y\|+1,
\end{equation}
from which
\begin{equation}\label{e:imagebound}
\|s^*\|\leq \|s^*-y^*\|+\|y^*\|\leq M + \|y\| + \|y^*\| + 1.
\end{equation}
Since $\rl(s - y,s^* - y^*) = \thalb\|s - y\|^2 + \scal{s - y}{s^*-y^*} + \thalb\|s^*-y^*\|^2$, \eqref{e:innerdiff}, \eqref{e:image} and \eqref{e:bound2} imply that  
\begin{align*}
\rl(s &- y,s^* - y^*)
\leq \scal{s - y}{t^* - y^*} - \thalb\|s - y\|^2 +\thalb(\|s - y\|+\beta)^2\\
&\leq \beta\|s - y\| - \thalb\|s - y\|^2 + \thalb\|s - y\|^2 + \beta\|s - y\| + \thalb\beta^2\\
&= 2\beta\|s - y\| + \thalb\beta^2.
\end{align*}
Thus, by \eqref{e:preimage},
\begin{align*}
\rl(s - y,s^* - y^*) \leq 2\beta (M + \|y\|) + \beta =\beta[2(M + \|y\|) + 1]<\varepsilon.
\end{align*}
From \eqref{e:bound2} and \eqref{e:imagebound}, $\|s\| \le M$ and $\|s^*\| \le M + \|y\| + \|y^*\| + 1$.   Consequently, $S$ is stably $\rl$--dense in $E\times E^*$.
\end{proof}
\begin{corollary}[Minty's condition for maximal monotonicity] Let $S:E\rightrightarrows E^*$ be monotone, $\gra S$ be closed in $E\times E^*$ and, for all $y\in E$, $S + J(\cdot-y)$ have hyperdense range.  Then $S$ is maximally monotone.
\end{corollary}
\begin{proof} This is immediate from Theorem~\ref{CLOSEDRLMAXthm} and Theorem~\ref{t:rloperator}.
\end{proof}
\section*{Acknowledgments}
XW was partially supported by a Discovery Grant of NSERC.

\end{document}